\documentclass[12pt, reqno]{amsart}

\usepackage{amsmath}
\usepackage{amsthm}
\usepackage{amssymb}
\usepackage{amsrefs}
\usepackage{mathrsfs}
\usepackage[usenames]{color}
\usepackage{bm}
\usepackage{enumerate}

\makeatletter
\newcommand{\leqnos}{\tagsleft@true\let\veqno\@@leqno}
\newcommand{\reqnos}{\tagsleft@false\let\veqno\@@eqno}
\reqnos
\makeatother

\newtheorem{thm}{}[section]
\newtheorem{theorem}[thm]{Theorem}
\newtheorem{corollary}[thm]{Corollary}
\newtheorem{lemma}[thm]{Lemma}

\theoremstyle{definition}
\newtheorem{definition}[thm]{Definition}
\theoremstyle{remark}

\newtheorem{question}[thm]{Question}

\numberwithin{equation}{section}
\allowdisplaybreaks

\newcommand{\psib}{\ensuremath{\bm{\psi}}}
\newcommand{\dd}{\ensuremath{\bm{d}}}

\newcommand{\Disc}{\ensuremath{\mathbb{D}}}
\newcommand{\sss}{\ensuremath{\bm{s}}}
\newcommand{\ww}{\ensuremath{\bm{w}}}
\newcommand{\bb}{\ensuremath{\bm{b}}}
\newcommand{\ee}{\ensuremath{\bm{e}}}
\newcommand{\xx}{\ensuremath{\bm{x}}}
\newcommand{\yy}{\ensuremath{\bm{y}}}
\newcommand{\zz}{\ensuremath{\bm{z}}}

\newcommand{\BB}{\ensuremath{\mathcal{B}}}
\newcommand{\GG}{\ensuremath{\mathcal{G}}}
\newcommand{\NN}{\ensuremath{\mathbb{N}}}
\newcommand{\LL}{\ensuremath{\mathbb{B}}}
\newcommand{\FF}{\ensuremath{\mathbb{F}}}
\newcommand{\XX}{\ensuremath{\mathbb{X}}}
\newcommand{\kk}{\ensuremath{k}}
\newcommand{\Fou}{\ensuremath{\mathcal{F}}}
\newcommand{\SSB}{\ensuremath{\mathcal{S}}}
\newcommand{\TTB}{\ensuremath{\mathcal{T}}}
\newcommand{\SA}{\ensuremath{\bm{A}}}
\newcommand{\UU}{\ensuremath{\mathcal{U}}}
\newcommand{\Sym}{\ensuremath{\mathbb{S}}}
\newcommand{\LO}{\ensuremath{\mathcal{L}}}
\newcommand{\Ha}{\ensuremath{\mathcal{H}}}
\newcommand{\SL}{\ensuremath{\mathscr{L}}}

\newcommand{\XXB}{\ensuremath{\mathcal{X}}}
\newcommand{\YY}{\ensuremath{\mathbb{Y}}}

\newcommand{\VV}{\ensuremath{\mathbb{V}}}
\newcommand{\Id}{\ensuremath{\mathrm{Id}}}

\DeclareMathOperator{\spn}{span}
\DeclareMathOperator{\sgn}{sign}
\DeclareMathOperator*{\Ave}{Ave}

\begin{document}

\title[Democracy of quasi-greedy bases in $\ell_{p}$ for $0<p< 1$]{Quasi-greedy bases in the spaces $\bm{\ell_{p}}$ ($\bm{0<p< 1}$)\\are democratic}

\author[F. Albiac]{Fernando Albiac}
\address{Mathematics Department-InaMat\\
Universidad P\'ublica de Navarra\\
Campus de Arrosad\'ia\\
Pamplona\\
31006 Spain}
\email{fernando.albiac@unavarra.es}

\author[J.~L. Ansorena]{Jos\'e L. Ansorena}
\address{Department of Mathematics and Computer Sciences\\
Universidad de La Rioja\\
Logro\~no\\
26004 Spain}
\email{joseluis.ansorena@unirioja.es}

\author[P. Wojtaszczyk]{Przemys\l{}aw Wojtaszczyk}
\address{Institute of Mathematics of the Polish Academy of Sciences\\
00-656 Warszawa\\
ul. \'Sniadeckich 8\\
Poland}
\email{wojtaszczyk@impan.pl}



\begin{abstract}
The list of known Banach spaces whose linear geometry determines the (nonlinear) democracy functions of their quasi-greedy bases to the extent that they end up being democratic, reduces to $c_0$, $\ell_2$, and all separable $\SL_1$-spaces. Oddly enough, these are the only Banach spaces that, when they have an unconditional basis, it is unique. Our aim in this paper is to study the connection between quasi-greediness and democracy of bases in nonlocally convex spaces. We prove that all quasi-greedy bases in $\ell_{p}$ for $0<p<1$ (which also has a unique unconditional basis) are democratic with fundamental function of the same order as $(m^{1/p})_{m=1}^\infty$. The methods we develop allow us to obtain even more, namely that the same occurs in any separable $\SL_{p}$-space, $0<p<1$, with the bounded approximation property.
\end{abstract}

\thanks{F. Albiac acknowledges the support of the Spanish Ministry for Economy and Competitivity under Grant MTM2016-76808-P for \emph{Operators, lattices, and structure of Banach spaces}. F. Albiac and J.~L. Ansorena acknowledge the support of the Spanish Ministry for Science, Innovation, and Universities under Grant PGC2018-095366-B-I00 for \emph{An\'alisis Vectorial, Multilineal y Aproximaci\'on}. P. Wojtaszczyk was supported by National Science Centre, Poland grant UMO-2016/21/B/ST1/00241.}


\maketitle

\section{Introduction}
\noindent The study of greedy-like bases from a functional analytic point of view sprang from the celebrated characterization of greedy bases in Banach spaces as those bases that are simultaneously unconditional and democratic \cite{KoTe1999}. Since Konyagin and Temlyakov's foundational result, several authors have considered derived forms of unconditionality and democracy which, either combined or separately, have given rise to new types of bases of interest both in approximation theory and in functional analysis. In this paper we are concerned with the possible connections between the properties of unconditionality and democracy (or some of its variations) in the general framework of quasi-Banach spaces.

We shall start by recalling the main concepts that we will need and setting the terminology.

Let $ \XXB=(\xx_{n})_{n=1}^\infty$ be a semi-normalized, fundamental, $M$-bounded Markushevich basis of a quasi-Banach space (in particular, a Banach space) $\XX$ over the real or complex field $\FF$ , i.e., $\XXB$ generates the whole space $\XX$ and there is a sequence $(\xx_{n}^*)_{n=1}^\infty$ in the dual space $\XX^{\ast}$ such that $(\xx_{n}, \xx_{n}^{\ast})_{n=1}^{\infty}$ is a biorthogonal system with $\inf_{n}\Vert \xx_{n}\Vert>0$ and $\sup_{n}\Vert \xx_{n}^{\ast}\Vert<\infty$. From now on we will refer to any such $\XXB$ simply as a \emph{basis}. A basic sequence will be a sequence in $\XX$ which is a basis of its closed linear span. Note that semi-normalized Schauder bases are a particular case of bases.

For a fixed sequence $\gamma=(\gamma_{n})_{n=1}^\infty\in\FF^\NN$, let us consider the map
\[
S_\gamma=S_\gamma[\XXB,\XX]\colon \spn( \xx_{n} \colon n\in\NN) \to \XX,
\quad \sum_{n=1}^\infty a_{n}\, \xx_{n} \mapsto \sum_{n=1}^\infty \gamma_{n}\, a_{n} \, \xx_{n}.
\]
The basis $\XXB$ is \emph{unconditional} if $S_\gamma$ is well-defined on $\XX$ for all $\gamma\in\ell_\infty$ and
\begin{equation}\label{eq:lu}
K_{u}=K_{u}[\XXB,\XX]:=\sup_{\Vert \gamma\Vert_\infty\le 1} \Vert S_\gamma\Vert <\infty.
\end{equation}
If $\XXB$ is unconditional, $K_{u}$ is called the unconditional basis constant. Now, given $A\subseteq \NN$, we define the \emph{coordinate projection} onto $A$ (with respect to the basis $\XXB$) as
\[
S_A=S_{\gamma_A}[\XXB,\XX],
\]
where $\gamma_A=(\gamma_{n})_{n=1}^\infty$ is the sequence defined by $\gamma_{n}=1$ if $n\in A$ and $\gamma_{n}=0$ otherwise. It is known (see, e.g., \cite{AABW2019}*{Theorem 1.10}) that
$\XXB$ is unconditional if and only if it is \emph{suppression unconditional}, i.e.,
\[
\sup \{ \Vert S_A\Vert \colon A\subseteq\NN \text{ finite}\}<\infty.
\]

Given a basis $\XXB=(\xx_{n})_{n=1}^\infty$ of a quasi-Banach space $\XX$, the \emph{coefficient transform}
\begin{equation*}\label{eq:Fourier}
\Fou\colon \XX \to \FF^\NN , \quad f\mapsto (\xx_{n}^*(f))_{n=1}^{\infty}
\end{equation*}
is a bounded linear operator from $\XX$ into $c_0$, hence for each $m\in\NN$ there is a unique $A=A_{m}(f)\subseteq \NN$ of cardinality $|A|=m$ such that whenever $n\in A$ and $k \in \NN\setminus A$, either $|a_{n}|>|a_{k}|$ or $|a_{n}|=|a_{k}|$ and $n<k$.
The \emph{$m$th greedy approximation} to $f\in\XX$ with respect to the basis $\XXB$ is
\[
\GG_{m}(f)=\GG_{m}[\XXB,\XX](f):=S_{A_{m}(f)}(f).
\]
Note that the operators $(\GG_{m})_{m=1}^{\infty}$ defining the \emph{greedy algorithm} on $\XX$ with respect to $\XXB$ are not linear nor continuous. The basis $\XXB$ is said to be \emph{quasi-greedy} if there is a constant $C\ge 1$ such that
\[
\Vert \GG_{m}(f)\Vert \le C\Vert f \Vert, \quad f\in\XX, \, m\in\NN.
\]
Equivalently, by \cite{Wo2000}*{Theorem 1} (see also \cite{AABW2019}*{Theorem 3.1}), these are precisely the bases for which the greedy algorithm merely converges, i.e.,
\[
\lim_{m\to\infty} \GG_{m}(f)=f \text{ for all } f\in \XX.
\]

Unconditional bases area special kind quasi-greedy bases, and although the converse is not true in general, quasi-greedy basis always retain in a certain sense a flavour of unconditionality. For example, they are \emph{suppression unconditional for constant coefficients} (or SUCC, for short), i.e., there is a constant $C\ge 1$ such that whenever $A$, $B$ are finite subsets of $\NN$ with $A\subseteq B$ and $(\varepsilon_{n})_{n\in B}$ are signs (i.e., scalars of modulus one) we have
\[
\left\Vert \sum_{n\in A} \varepsilon_{n} \, \xx_{n} \right\Vert \le C \left\Vert \sum_{n\in B} \varepsilon_{n} \, \xx_{n} \right\Vert
\]
(see \cite{AABW2019}*{Lemma 2.2} and \cite{Wo2000}*{Proposition 2}).
If the basis is SUCC then there is another constant $C\ge 1$ such that
\begin{equation}\label{eq:succ}
\left\Vert \sum_{n\in A} \theta_{n} \, \xx_{n} \right\Vert \le C \left\Vert \sum_{n\in A} \varepsilon_{n} \,\xx_{n}\right\Vert
\end{equation}
for all finite subsets $A$ of $\NN$ and all choice of signs $(\theta_{n})_{n\in A}$ and $(\varepsilon_{n})_{n\in A}$.

In turn, a basis $\XXB=(\xx_{n})_{n=1}^\infty$ of a quasi-Banach space $\XXB$ is said to be \emph{democratic} if blocks of $\XXB$ of the same size have uniformly comparable quasi-norms, i.e., there is a constant $D\ge 1$ such that
\[
\left\Vert\sum_{n\in A} \xx_{n}\right\Vert\le D \left\Vert\sum_{n\in B} \xx_{n}\right\Vert,
\]
for any two finite subsets $A$, $B$ of $\NN$ with $|A|=|B|$. The lack of democracy of a basis $\XXB$ exhibits some sort of asymmetry. To measure how much a basis $\XXB$ deviates from being democratic, we consider its \emph{upper democracy function}, also known as its \emph{fundamental function},
\[
\varphi_{u}[\XXB, \XX](m): = \varphi_{u}(m)=\sup_{|A|\le m}\left\|\sum_{n\in A}{\xx_{n}}\right\|,\qquad m=1,2,\dots,
\]
and its \emph{lower democracy function},
\[
\varphi_{l}[\XXB, \XX](m):= \varphi_{l}(m)=\inf_{|A|\ge m}\left\|\sum_{n\in A} \xx_{n} \right\|, \qquad m=1,2,\dots.
\]

If $\XXB$ is SUCC then $\varphi_{l}(m)\lesssim\varphi_{u}(m)$ for $m\in\NN$, hence $\XXB$ is democratic if and only $\varphi_{u}(m)\lesssim\varphi_{l}(m)$ for $m\in\NN$. Moreover, for any set $A$ with $|A|=m$ we have
\begin{equation}\label{eq:LDA}
\inf_{|A|= m}\left\|\sum_{n\in A} \xx_{n} \right\|
\lesssim \varphi_{l}[\XXB, \XX](m), \quad m\in\NN,
\end{equation}
in which case it is \emph{super-democratic}, i.e., there is a constant $D\ge 1$ such that
\begin{equation*}
\left\Vert \sum_{n\in A} \theta_{n}\, \xx_{n} \right\Vert \le D \left\Vert \sum_{n\in B} \varepsilon_{n}\, \xx_{n} \right\Vert
\end{equation*}
for any two finite subsets $A$, $B$ of $\NN$ with $|A|=|B|$, and any signs $(\theta_{n})_{n\in A}$ and $(\varepsilon_{n})_{n\in B}$.
Here, and throughout this paper, the symbol $\alpha_{j}\lesssim \beta_{j}$ for $j\in J$ means that there is a positive constant $C<\infty$ such that the families of non-negative real numbers $(\alpha_{j})_{j\in J}$ and $(\beta_{j})_{j\in J}$ are related by the inequality $\alpha_{j}\le C\beta_{j}$ for all $j\in J$. If $\alpha_{j}\lesssim \beta_{j}$ and $\beta_{j}\lesssim \alpha_{j}$ for $j\in J$ we say $(\alpha_{j})_{j\in J}$ are $(\beta_{j})_{j\in J}$ are equivalent, and we write $\alpha_{j}\approx \beta_{j}$ for $j\in J$.

\subsection*{Quasi-greedy vs.\@ democratic bases} In general, quasi-greedy (or even unconditional) bases need not be democratic and, conversely, democratic bases may not be quasi-greedy. Thus, these two properties are a priori independent of each other and they can be thought of as the two pillars that sustain the entire theory of greedy approximation using bases. Indeed, apart from the aforementioned characterization of greedy bases in terms of unconditionality and democracy, this claim is supported by the characterization of almost greedy basis as those bases that are at the same time quasi-greedy and democratic \cite{DKKT2003}.

In order to investigate the connection between quasi-greediness and democracy, it is very natural to ask in which way the geometry of the space affects the democracy functions of quasi-greedy bases. For instance, although the democracy functions $\varphi_{l}[\XXB, \XX]$ and $\varphi_{u}[\XXB, \XX]$ may vary as we consider different quasi-greedy bases $\XXB$ within the same space $\XX$, there exist spaces for which all quasi-greedy bases have essentially the same democracy functions. The first result in this direction appeared in \cite{Wo2000} where it was proved that for any quasi-greedy basis $\BB$ in $\ell_{2}$ we have
\begin{equation}\label{eq:Deml2}
\varphi_{l}[\XXB,\ell_2](m)\approx m^{1/2}\approx \varphi_{u}[\XXB,\ell_2](m),\quad m\in \NN,
\end{equation}
so that all quasi-greedy bases of $\ell_{2}$ are democratic.

Subsequently, Dilworth el al.\@ proved that the unit vector system is, up to equivalence, the unique quasi-greedy basis of $c_0$ (see \cite{DKK2003}*{Corollary 8.6}), hence formally speaking all quasi-greedy bases $\XXB$ in $c_{0}$ are democratic with $\varphi_{u}[\XXB,c_{0}]\approx\varphi_{l}[\XXB, c_{0}]\approx 1$. In this case even more can be said, namely that $c_0$ is the unique Banach space whose dual is a GT space and has a quasi-greedy basis (\cite{DKK2003}*{Proposition~8.1}). In particular, $c_0$ is the unique $\SL_\infty$-space with a quasi-greedy basis. In this line of thought, Dilworth et al.\@ \cite{DSBT2012} achieved the following result, which applies, in particular to $\ell_1$ and $L_1$.
\begin{theorem}[\cite{DSBT2012}*{Theorem 4.2}]\label{thm:GTDemocratic} Suppose that $\XXB$ is a quasi-greedy basis of a GT space $\XX$. Then $\XXB$ is democratic with
\[
\varphi_{l}[\XXB,\XX](m)\approx m \approx\varphi_{u}[\XXB,\XX](m), \quad m\in\NN.
\]
\end{theorem}
The non-specialist reader will find in the Appendix (see Section~\ref{Sect:DemL1}) the necessary information on $\SL_{p}$-spaces, $1\le p\le\infty$, and GT spaces.

Let us next summarize the interplay between quasi-greediness and democracy for bases in the spaces $\ell_{p}$ and $L_{p}=L_{p}([0,1])$ for $1<p<\infty$, $p\neq 2$. Despite the fact that any super-democratic
(in particular, democratic and quasi-greedy) basis of $\ell_{p}$ satisfies
\begin{equation*}
\varphi_{l}[\XXB,\ell_{p}](m) \approx m^{1/p}\approx \varphi_{u}[\XXB,\ell_{p}](m),\quad m\in \NN,
\end{equation*}
(see \cite{AlbiacAnsorena2015}*{Corollary 2.7}), $\ell_{p}$ possesses quasi-greedy bases that are not democratic. Roughly speaking this could be interpreted by saying that the linear structure of $\ell_{p}$ for $p\not= 2, 1$ is less restrictive on the (nonlinear) democracy functions of quasi-greedy bases of the space, and in fact the only geometric features that shed any information in this respect are the Rademacher type and cotype. Indeed, a similar argument to the one used in \cite{Wo2000} to obtain \eqref{eq:Deml2} yields that all SUCC bases $\XXB$ of a quasi-Banach $\XX$ with type $0<q\le 2$ and cotype $r\ge 2$ satisfy
\begin{equation}\label{eq:RademacherEstimates}
m^{1/r} \lesssim \varphi_{l}[\XXB,\XX](m), \quad\text{and}\quad \varphi_{u}[\XXB,\XX](m) \lesssim m^{1/q}, \quad m\in\NN
\end{equation}
(cf. \cite{AlbiacAnsorena2015}*{Lemma 2.5}). In the case when $\XX=\ell_{p}$ or $\XX=L_{p}$, $1<p<\infty$, these estimates are sharp. Indeed, it is well-known (see (\cite{Pel1960}) that for $1<p<\infty$, the space $\ell_{p}$ is isomorphic to $\XX_{p}=(\bigoplus_{n=1}^{\infty}\ell_2^n)_{\ell_{p}}$. Hence, the canonical basis of $\XX_{p}$ provides (through the isomorphism) an unconditional, hence quasi-greedy, basis of $\ell_{p}$ with $\varphi_{l}(m)\approx m^{1/r}$ and $\varphi_{u}(m)\approx m^{1/q}$ for $m\in\NN$, where $r=\min\{p,2\}$ is the optimal type of $\ell_{p}$ and $q=\max\{p,2\}$ is its optimal cotype.

As far the space $L_{p}([0,1])$ for $1<p<\infty$, $p\not=2$, is concerned we point out that, unlike $\ell_{p}$, this space possesses democratic quasi-greedy bases with different fundamental functions. To see this it is convenient to recall the following result.
\begin{theorem}[see \cite{Nielsen2007}*{Theorem 1.4} and \cite{DSBT2012}*{Theorem 1.4}]\label{thm:UBQGOrtho}
There is an orthogonal system $\Psi=(\psib_{n})_{n=1}^\infty$ in $L_2$ with $\sup_{n} \Vert \psib_{n}\Vert_\infty<\infty$ such that $\Psi$ is a quasi-greedy basis of $L_{p}$ for each $1<p<\infty$.
\end{theorem}
Now, on one hand, if $1<p<\infty$ and $\Psi$ is as in Theorem~\ref{thm:UBQGOrtho}, by \cite{AACV2019}*{Proposition 2.5} we have
\[
\varphi_{l}[\Psi,L_{p}](m)\approx m^{1/2} \approx \varphi_{m}[\Psi,L_{p}](m), \quad m\in\NN.
\]
On the other hand, the $L_{p}$-normalized Haar system $\Ha^{(p)}$ is a unconditional and democratic basis of $L_{p}$ with
\[
\varphi_{l}[\Ha^{(p)},L_{p}](m)\approx m^{1/p} \approx \varphi_{m}[\Ha^{(p)},L_{p}](m), \quad m\in\NN
\]
(see \cite{Temlyakov1998}). In contrast, since any unconditional basis of $L_{p}$ possesses a subbasis equivalent to the unit vector system of $\ell_{p}$ (see \cite{KadecPel1962}), any democratic unconditional basis $\XXB$ of $L_{p}$ satisfies
\begin{equation*}
\varphi_{l}[\XXB,L_{p}](m) \approx m^{1/p}\approx \varphi_{u}[\XXB,L_{p}](m),\quad m\in\NN.
\end{equation*}
Since $\ell_2$ is a complemented subspace of $L_{p}$, applying \cite{GHO2013}*{Proposition 6.1} yields that the direct sum of $\Ha^{(p)}$ and the unit vector system of $\ell_2$ is an unconditional basis (hence, quasi-greedy) basis of (a space isomorphic to) $L_{p}$ with $\varphi_{l}(m)\approx m^{1/r}$ and $\varphi_{u}(m)\approx m^{1/q}$ for $m\in\NN$. Let us also mention that every $\SL_{p}$-space $\XX$ other than $\ell_{p}$ has a democratic quasi-greedy basis with fundamental function equivalent to $(m^{1/s})_{m=1}^\infty$ for $s\in\{2,p\}$ (see \cite{AADK2019}*{Example 4.6}).

The above examples show that the connection between democracy and quasi-greediness of bases in $\SL_{p}$-spaces for $1\le p \le \infty$ is by now completely understood. The attentive reader might have noticed a pattern here, namely that the only indices $p\in [1,\infty]$ for which all quasi-greedy bases of $\ell_{p}$ (with the convention that $\ell_{\infty}$ means $c_0$) are democratic coincide with the values of $p$ for which $\ell_{p}$ has a unique unconditional basis (\cites{KT1934,LinPel1968, LindenstraussZippin1969}).

Motivated by those results, and also by the recent nontrivial extension to (not necessarily locally convex) quasi-Banach spaces of the characterization of almost greedy bases as democratic and quasi-greedy (see \cite{AABW2019}*{Theorem 5.3}), in this article we initiate the study of the connection between quasi-greediness and democracy of bases in the lack of local convexity of the underlying space. Since $L_{p}([0,1])$ for $0<p<1$ has trivial dual (making it therefore impossible for $L_{p}$ to have a basis), the first nonlocally convex spaces that come to mind as objects of study are the spaces $\ell_{p}$ for $0<p<1$. Kalton proved that these spaces also have a unique unconditional basis (see \cite{Kalton1977}), hence it seems reasonable to conjecture that quasi-greedy bases in $\ell_{p}$ for $0<p<1$ will follow the pattern of quasi-greedy bases in $\ell_{p}$ for $p=1,2,\infty$, and will end up being democratic. Our guess was reinforced by the results obtained in the recent paper \cite{AAW2020}, where the authors construct a continuum of mutually permutatively nonequivalent quasi-greedy bases in each $\ell_{p}$, and all of them are democratic. The main result of the present paper consists of confirming our conjecture by showing the following theorem.

\begin{theorem}\label{MainThm} Let $0<p<1$. If $\XXB$ is a quasi-greedy basis of $\ell_{p}$ then
\[
\varphi_{l}[\XXB,\ell_{p}](m) \approx m^{1/p}\approx \varphi_{u}[\XXB,\ell_{p}](m),\quad m\in \NN.
\]
In particular, $\XXB$ is democratic.
\end{theorem}
Note that the $p$-convexity of the space immediately yields that any basis $\XXB$ of a $p$-Banach space $\XXB$, $0<p\le 1$, satisfies
\begin{equation}\label{eq:ObviouspBanach}
\varphi_{u}[\XXB,\XX](m)\lesssim m^{1/p},\quad m\in\NN.
\end{equation}
So, the challenge with Theorem~\ref{MainThm} consists on developing the specific tools that permit to show that when $0<p<1$, $m^{1/p} \lesssim \varphi_{l}[\XXB,\ell_{p}](m)$ for $m\in\NN$. We will take care of this in Section~\ref{sect:main}. Prior to that, for the reader's sake, in Section~\ref{prelimSect} we gather the most relevant preliminary results. We close with an Appendix
mainly devoted to providing a simplified and more direct proof of Theorem~\ref{thm:GTDemocratic}. This last section has an educational purpose and exhibits once again the fact that the methods used for the case $p=1$ are rendered useless when the local convexity of the space is lifted.

Throughout this paper we use standard facts and notation from Banach spaces and approximation theory (see e.g.\@ \cite{AlbiacKalton2016}). The reader will find the required specialized background and notation on greedy-like bases in quasi-Banach spaces in \cite{AABW2019}.

\section{Preliminaries}\label{prelimSect}

\noindent A family of nonlinear operators of key relevance in the study of the convergence of the greedy algorithm $(\mathcal G_{m})_{m=1}^{\infty}$ in a quasi-Banach space $\XX$ with respect to a basis $\XXB=(\xx_{n})_{n=1}^\infty$is the sequence $(\UU_{m})_{m=1}^{\infty}$ of restricted truncation operators, defined as follows. For $m\in \NN$, the $m$\emph{th-restricted truncation operator} $\UU_{m}\colon \XX \to \XX$ is the map
\[
\UU_{m}(f)=\UU(f,A_{m}(f)), \quad f\in\XX,
\]
where
for each $f\in \XX$ and each $A\subseteq\NN$ finite,
\[
\UU(f,A) = \min_{n\in A} |\xx_{n}^*(f)| \sum_{n\in A} \sgn (\xx_{n}^*(f)) \, \xx_{n}.
\]
Here, as is customary, $\sgn(\cdot)$ denotes the sign function, i.e., $\sgn(0)=1$ and $\sgn(a)=a/|a|$ if $a\in\FF\setminus\{0\}$.

A crucial property for our purposes is that quasi-greedy bases in quasi-Banach spaces have the \emph{bounded restricted truncation operator property}, i.e., the operators $(\UU_{m})_{m=1}^{\infty}$ are uniformly bounded (\cite{AABW2019}*{Theorem 3.13}).

Since a basis $\XXB$ is unconditional if and only if it is suppression unconditional, to quantify the conditionality of a quasi-greedy basis in a quasi-Banach space $\XX$ we use the \emph{conditionality constants} of the basis,
\[
\kk_{m}[\XXB,\XX] =\sup_{|A|\le m} \Vert S_A[\XXB,\XX]\Vert,\quad m=1,2,\dots.
\]
If $\XX$ is a $p$-Banach space, then the $p$-convexity of the space immediately yields $\kk_{m}[\XXB,\XX]\lesssim m^{1/p}$ for $m\in\NN$, and this is the best one can hope for in general. Indeed, the difference basis, $\dd_{n}=\ee_{n}-\ee_{n-1}$ for $n=2$, $3,\dots$ and $\dd_1=\ee_1$, of $\ell_{p}$ satisfies $\kk_{m}[\XXB,\XX] \ge (2m)^{1/p}$ for each $m$. However, when the basis is quasi-greedy the size of the members of the sequence $(\kk_{m}[\XXB,\XX])_{m=1}^{\infty}$ is controlled by a slowly growing function:

\begin{theorem}[see \cite{AAW2020}*{Theorem 5.3}]\label{thm:EstimateCC}Let $\XX$ be a $p$-Banach space, $0<p\le 1$. Suppose $\XXB$ is a basis of $\XX$ with the bounded restricted truncation operator property. Then
\begin{equation*}
\kk_{m}[\XXB,\XX]\lesssim (1+\log (m))^{1/p}, \quad m\in\NN.
\end{equation*}
\end{theorem}

Let us next recall the basic ingredients and facts that we will need about embeddings via bases. Loosely speaking, this method aims at obtaining qualitative estimates on the symmetry of bases in $\XX$ by squeezing the space $\XX$ in between two symmetric sequence spaces that are sufficiently close to each other.

A \emph{symmetric sequence space} will be a subset $\Sym\subseteq\FF^\NN$ equipped with a ``gauge'' $\Vert \cdot\Vert_\Sym\colon \FF^\NN\to[0,\infty]$ such that
\begin{enumerate}[(q1)]
\item $\Vert f\Vert_\Sym >0$ for all $f\not=0$;
\item $\Vert t \, f\Vert_\Sym=|t| \, \Vert f\Vert_\Sym$ for all $t\in\FF$ and all $f\in \FF^\NN$;
\item $\Vert (b_{j})_{j=1}^\infty \Vert_\Sym\le \Vert (a_{j})_{j=1}^\infty\Vert_\Sym$ whenever $|b_{j}|\le |a_{j}|$ for every $j\in\NN$;
\item $\Vert \sum_{j\in A} \ee_{n} \Vert_\Sym<\infty$ for every $A\subseteq\NN$ finite;
\item if the sequence $(a_{j,k})_{j,k\in\NN} \subseteq [0,\infty)$ is non-decreasing in $k$, then
\[
\left\Vert \left(\lim_{k} a_{j,k}\right)_{j=1}^\infty\right\Vert_\Sym=\lim_{k} \Vert (a_{j,k})_{j=1}^\infty\Vert_\Sym;
\]
\item $\Vert (a_{\pi(j)})_{j=1}^\infty\Vert_\Sym=\Vert (a_{j})_{j=1}^\infty\Vert_\Sym$ for every permutation $\pi$ of $\NN$;
\item $\Sym=\{ f\in \FF^\NN \colon \Vert f \Vert_\Sym<\infty\}$.
\end{enumerate}

Let $\Sym$ be a symmetric sequence space and $\XX$ be a quasi-Banach space with a basis $\XXB$. Let us denote by $\Fou$ the coefficient transform with respect to $\XXB$. The following terminology was introduced in \cite{AlbiacAnsorena2015}.

\begin{enumerate}[(a)]
\item We say that \emph{$\Sym$ embeds in $\XX$ via $\XXB$}, and put $\Sym\stackrel{\XXB}\hookrightarrow \XX$, if there is a constant $C$ such that for every $g\in\Sym$ there is $f\in\XX$ such that $\Fou(f)=g$, and we have $\Vert f \Vert \le C\Vert g \Vert_{\Sym}$.

\item We say that \emph{$\XX$ embeds in $\Sym$ via $\XXB$}, and put $\XX\stackrel{\XXB}\hookrightarrow\Sym$, if there is a constant $C$ such that $\Fou(f)\in \Sym$ with $\Vert \Fou(f) \Vert_{\Sym}\le C \Vert f\Vert$ for all $f\in \XX$.
\end{enumerate}

Given a sequence $(a_{n})_{n=1}^\infty$ in $\FF$ with $\lim_{n} a_{n}=0$, the non-increasing rearrangement of $(|a_{j}|)_{j=1}^\infty$ will be denoted by $(a_{n}^*)_{n=1}^\infty$. Let $\ww=(w_{n})_{n=1}^\infty$ be a sequence of positive numbers, and let $\sss=(s_{n})_{n=1}^\infty$ be its \emph{primitive weight} defined by $s_{n}=\sum_{k=1}^n w_{k}$ for $n\in\NN$. For $0< q\le\infty$, the \emph{weighted Lorentz sequence space} $d_{q}(\ww)$ is the symmetric sequence space associated to the gauge $\Vert\cdot\Vert_{q,\ww}$ defined for $f=(a_{j})_{j=1}^\infty$ by
\[
\Vert f\Vert_{q,\ww}= \left(\sum_{n=1}^\infty (a_{n}^*)^q s_{n}^{q-1} w_{n}\right)^{1/q}, \quad\text{if }0<q<\infty,
\]
and
\[
\Vert f\Vert_{\infty,\ww}=\sup_{n} s_{n} a_{n}^*, \quad\text{if } q=\infty.
\]
If $\sss=(s_{n})_{n=1}^\infty$ is an increasing weight we denote by $\Delta\sss=(s_{n}-s_{n-1})_{n=1}^\infty$ its \emph{difference weight} (with the convention that $s_0=0$). There is an obvious relation between primitive and difference weights: if $\ww=\Delta\sss$ then $\sss$ is the primitive weight of $\ww$. Note that, if $\sss_\alpha=(n^\alpha)_{n=1}^\infty$ then $d_{q}(\Delta\sss_{1/p})=\ell_{p,q}$, for all $0<p<\infty$ and $0<q\le\infty$.

The following lemma gathers the connection between greedy-like bases and embeddings involving sequence Lorentz spaces.

\begin{lemma}[see \cite{AABW2019}*{Theorem 8.12 and Corollary 8.13}]\label{lem:DemocracyEmdedding} Let $\XXB$ be a basis of a $q$-Banach space $\XX$, $0<q\le 1$. Let $\ww=$ be a weight with primitive weight $\sss=(s_{n})_{n=1}^\infty$.

\begin{enumerate}[(a)]
\item\label{DemocracyEmdedding:item:1} Suppose the $\XXB$ has the bounded restricted truncation operator property. Then $\XX \stackrel{\XXB}\hookrightarrow d_\infty(\ww)$ if and only if
$
s_{m} \lesssim \varphi_{l}[\XXB,\XX](m)
$
for $m\in\NN$.
\item $d_{q}(\ww) \stackrel{\XXB}\hookrightarrow \XX $ if and only if
$
\varphi_{u}[\XXB,\XX](m) \lesssim s_{m}
$
for $m\in\NN$.
\end{enumerate}
\end{lemma}

Finally, we recall that if $\XX$ is a quasi-Banach space of type $p>1$ then $\XX$ is isomorphic to a Banach space \cite{Kalton1980}*{Theorem 4.1} while if the type of $\XX$ is $p<1$ then $\XX$ is isomorphic to a $p$-Banach space \cite{Kalton1980}*{Theorem 4.2}. Combining Lemma~\ref{lem:DemocracyEmdedding} with the estimates in \eqref{eq:RademacherEstimates} we obtain that if $\XX$ is a quasi-Banach space of type $p$ and cotype $r$, and $\XXB$ is a basis of $\XX$ with the bounded restricted truncation operator property then
\[
\ell_{p,q} \stackrel{\XXB}\hookrightarrow \XX \stackrel{\XXB}\hookrightarrow \ell_{r,\infty},
\]
where $q=1$ if $p>1$, $q=p$ if $p<1$, and $0<q<1$ is arbitrary if $p=1$. In the case when $p=r=2$ (i.e., $\XX$ is a Hilbert space \cite{Kwapien1972}),
we obtain
\begin{equation}\label{eq:HilbertEmbeddings}
\ell_{2,1} \stackrel{\XXB}\hookrightarrow \XX \stackrel{\XXB}\hookrightarrow \ell_{2,\infty}
\end{equation}
(cf. \cite{Wo2000}*{Theorem 3.1}).

\section{Quasi-greedy basis in $\SL_{p}$-spaces, $0<p<1$}\label{sect:main}
\noindent
Our approach towards proving Theorem~\ref{MainThm} is inspired by the techniques from \cite{Wo1997}, where it was shown that if a quasi-Banach space has a strongly absolute basis and is isomorphic to its square then it has a unique unconditional basis up to a permutation. Let us record this important definition for further reference.

\begin{definition}\label{def:sa} A semi-normalized unconditional basis $\BB=(\bb_{j})_{j\in J}$ of a quasi-Banach space $\LL$ is said to be \emph{strongly absolute} if for every $\epsilon>0$ there is a (smallest) constant $\SA(\epsilon)>0$ such that
\begin{equation}\label{eq:sb}
\sum_{j\in J} |\bb_{j}^*(f)| \le \max\left\{\SA(\epsilon) \sup_{j\in J} |\bb_{j}^*(f)| , \epsilon \Vert f \Vert \right\}, \quad f\in\XX.
\end{equation}
In this case, the map $\SA\colon(0,\infty)\to(0,\infty)$ will be called the \emph{strongly absolute function} of $\BB$.
\end{definition}
\noindent
Strongly absolute bases were introduced in \cite{KLW1990} to study the uniqueness of unconditional structure of nonlocally convex quasi-Banach spaces. Roughly speaking, a basis is strongly absolute if it dominates the unit vector system of $\ell_1$ while remaining far from it. This is the case with the unit vector system $(\ee_{j})_{j=1}^\infty$ when regarded as a basis of $\ell_{p}$ for $0<p<1$. Recall that for $j\in\NN$, $\ee_{j}=(\delta_{i,j})_{i=1}^\infty$, where $\delta_{i,j}=1$ if $i=j$ and $\delta_{i,j}=0$ otherwise. The vectors $(\ee_{j})_{j=1}^\infty$ form a normalized $1$-unconditional basis of $\ell_{p}$ whose associated sequence $(\ee_{j}^*)_{j=1}^\infty$ of biorthogonal functionals satisfies $f=(\ee_{j}^*(f))_{j=1}^\infty$ for all $f\in\ell_{p}$.

\begin{lemma}[cf.\@ \cite{Leranoz}*{Lemma 2.2}]\label{lem:SAlp} For $0<p<1$, the unit vector system is a strongly absolute basis of $\ell_{p}$ whose strongly absolute function $\SA$ satisfies
\[
\SA(\epsilon)\le \epsilon^{-p/(1-p)}, \quad \epsilon>0.
\]
\end{lemma}
\begin{proof} Given $\epsilon>0$, set $C=\epsilon^{-p/(1-p)}$.
For $f\in\ell_{p}$ we have
\begin{align*}
\sum_{j=1}^\infty |\ee_{j}^*(f)| &=\Vert f \Vert_1\\
&\le \Vert f\Vert_\infty^{1-p} \, \Vert f\Vert_{p}^p\\
&= \left( C \Vert f\Vert_\infty\right)^{1-p} \,
\left( \epsilon \Vert f\Vert_{p}\right)^p\\
&\le \max \{ C \Vert f \Vert_\infty, \epsilon \Vert f\Vert_{p} \}\\
&= \max \left\{ C \sup_{j\in\NN} |\ee_{j}^*(f)|, \epsilon \Vert f\Vert_{p} \right\}.
\qedhere
\end{align*}
\end{proof}

Suppose $\XX$ is a quasi-Banach space with an unconditional basis $(\bb_{j})_{j\in J}$. Given a family $\SSB=(\xx_{n},\xx_{n}^*)_{i\in A}$ in $\XX\times\XX^*$, and $\delta>0$ we consider the following sets of indices from $J$:
\begin{equation}\label{eq:Wo}
\Omega_\delta(\SSB)=\{ j \in J \colon |\xx_{n}^*(\bb_{j}) \, \bb_{j}^*(\xx_{n})|\ge \delta \text{ for some } n\in A\}.
\end{equation}

The analysis of the sets $\Omega_\delta(\SSB)$ for families $\SSB$ associated to unconditional bases
was successfully used in \cite{Wo1997} to advance the research initiated in \cite{KLW1990} on uniqueness of unconditional bases in quasi-Banach spaces. The following lemma is in the spirit of the estimates obtained in \cite{Wo1997}. In the absence of unconditionality, the role played by Lemma~\ref{lem:WoIdea} in the proof of Theorem~\ref{MainThm} runs parallel to the role played in the proof of Theorem~\ref{thm:GTDemocratic} by the interpretation of Grothendieck's inequality of Lindenstrauss and Pe{\l}czy{\'n}ski in their proof of the uniqueness of unconditional basis of $\ell_1$ (see Section~\ref{Sect:DemL1}).

\begin{lemma}\label{lem:WoIdea} Let $\LL$ be a quasi-Banach space with a strongly absolute basis $\BB=(\bb_{j})_{j\in J}$ and $\SSB=(\xx_{n},\xx_{n}^*)_{i\in A}$ be a finite family in $\LL\times\LL^*$ such that $\xx_{n}^*(\xx_{n})=1$, $ \Vert \xx_{n}\Vert\le a$ and $\Vert \xx_{n}^*\Vert\le b$ for all $n\in A$. Then for each $C\in(1,\infty)$, there is $\delta>0$ such that
\[
|A|\le C \sum_{j\in \Omega_\delta(\SSB)} \left| \sum_{n\in A} \xx_{n}^*(\bb_{j}) \, \bb_{j}^*(\xx_{n})\right|.
\]
Moreover, if $c=\sup_{j} \Vert \bb_{j} \Vert$, $K_{u}$ is the unconditional basis constant of $\BB$, and $\SA$ is the strongly absolute function of $\BB$, we can choose
\[
\delta= \frac{C-1}{C}\frac{1}{ \SA(\epsilon) }, \quad { where }\quad \epsilon=\frac{C-1}{C} \frac{1}{abcK_{u}}.
\]
\end{lemma}

\begin{proof}For $j\in J$ and $n\in A$ set
\begin{align*}
\lambda_{j}&=\sum_{n\in A} \xx_{n}^*(\bb_{j}) \, \bb_{j}^*(\xx_{n}),\\
f_{n}&= \sum_{j\in J\setminus \Omega_\delta(\SSB)} \xx_{n}^*(\bb_{j}) \, \bb_{j}^*(\xx_{n}) \, \bb_{j}.
\end{align*}
By construction, $|\bb_{j}^*(f_{n})|\le\delta$ for all $n\in A$ and $j\in J$, and by unconditionality, $\Vert f_{n}\Vert \le bcK_{u} \Vert \xx_{n}\Vert$ for all $n\in\NN$. Then,
\begin{align*}
|A|&=\left| \sum_{n\in A} \xx_{n}^*(\xx_{n}) \right|\\
&=\left| \sum_{n\in A} \sum_{j\in J} \xx_{n}^*(\bb_{j}) \, \bb_{j}^*(\xx_{n}) \right|\\
&\le\left| \sum_{n\in A} \sum_{j\in J\setminus \Omega_\delta(\SSB)} \xx_{n}^*(\bb_{j}) \, \bb_{j}^*(\xx_{n}) \right|
+ \sum_{n\in A} \left|\sum_{j\in \Omega_\delta(\SSB)} \xx_{n}^*(\bb_{j}) \, \bb_{j}^*(\xx_{n}) \right|\\
&= \sum_{n\in A} \sum_{j\in J} | \bb_{j}^*(f_{n})| + \sum_{j\in \Omega_\delta(\SSB)} |\lambda_{j}| \\
&\le \sum_{n\in A}
\max\left\{ \SA(\epsilon) \delta, \epsilon \Vert f_{n} \Vert \right\}
+ \sum_{j\in \Omega_\delta(\SSB)} |\lambda_{j}|\\
&\le \sum_{n\in A} \max\{ \SA(\epsilon)\delta, b c K_{u} \epsilon \Vert \xx_{n} \Vert \} + \sum_{j\in \Omega_\delta(\SSB)} |\lambda_{j}|\\
&= \frac{C-1}{C} |A| + \sum_{j\in \Omega_\delta(\SSB)} |\lambda_{j}|.
\end{align*}
Hence, $C^{-1}|A|\le \sum_{j\in \Omega_\delta(\SSB)} |\lambda_{j}|$.
\end{proof}

The following square-function estimate for vectors with constant coefficients is valid for quasi-greedy bases and, in lack of unconditionality, serves as a substitute of the Littlewood-Paley formula for unconditional bases in $L_{p}$-spaces.

\begin{lemma}\label{lem:LP} Let $0<p<1$. Suppose $(\xx_{n})_{i=1}^{\infty}$ is a SUCC basic sequence in $\ell_{p}$. Then
\[
\left\|\sum_{n\in A} \xx_{n}\right\|_{p}\approx \left(\sum_{j=1}^\infty \left( \sum_{n\in A} |\ee_{j}^*(\xx_{n})|^2\right)^{p/2}\right)^{1/p}
\]
for all $A\subseteq \NN$ finite.
\end{lemma}

\begin{proof} Using inequality \eqref{eq:succ} and the classical Khintchine's inequality (see e.g. \cite{AlbiacKalton2016}) we obtain
\begin{align*}
\left\|\sum_{n\in A} \xx_{n}\right\|_{p}^p
&\approx\Ave_{ \varepsilon_{n}=\pm 1} \left\|\sum_{n\in A}\varepsilon_{n} \, \xx_{n}\right\|_{p}^p\\
&=\Ave_{\varepsilon_{n}=\pm 1} \sum_{j=1}^\infty \left|\ee_{j}^*\left(\sum_{n\in A}\varepsilon_{n}\, \xx_{n}\right)\right|^p\\
&=\sum_{j=1}^\infty \Ave_{\varepsilon_{n}=\pm 1} \left| \sum_{n\in A} \varepsilon_{n}\, \ee_{j}^*(\xx_{n})\right|^p\\
&\approx \sum_{j=1}^\infty \left( \sum_{n\in A} |\ee_{j}^*(\xx_{n})|^2\right)^{p/2}.\qedhere
\end{align*}
\end{proof}

We are almost ready to tackle the proof of Theorem~\ref{MainThm}. Actually Theorem~\ref{MainThm} will follow as a particular case of Theorem~\ref{thm:main}, where we prove our main result for subspaces of $\ell_{p}$ that admit the extension to the whole $\ell_{p}$ of compact operators mapping into quasi-Banach spaces. This ideas go back to \cite{Lind1964}*{Theorem 2.2} and rely upon the following definition.

\begin{definition} Suppose $\XX$ and $\LL$ are quasi-Banach spaces and that $\XX$ is a subspace of $\LL$. We shall say that $\XX$ has the \emph{compact extension property in $\LL$} if every compact operator $S\colon\XX\to\YY$ mapping into a quasi-Banach space $\YY$ extends to a compact operator $T\colon\LL\to\YY$. \end{definition}
The next theorem follows from the open mapping theorem (see comments previous to \cite{Kalton1984}*{Theorem 2.2}).
\begin{theorem}\label{thm:CEP} Let $\XX$ and $\LL$ be quasi-Banach spaces such that $\XX$ is a subspace of $\LL$, and let $0<p\le 1$. If $\XX$ has the compact extension property in $\LL$ then there is a constant $D\ge 1$ such that any compact operator $S\colon \XX\to \YY$ mapping into a $p$-Banach space $\YY$ extends to a compact operator $T\colon\LL\to\YY$ with $\Vert T\Vert\le D \Vert S\Vert$.
\end{theorem}

\begin{theorem}\label{thm:main} Let $\XX$ be a closed subspace of $\ell_{p}$, $0<p<1$, and suppose that $\XX$ has the compact extension property in $\ell_{p}$. If $\XXB$ is a basis of $\XX$ with the bounded restricted truncation operator property, then $\XXB$ is democratic with
\[
\varphi_{l}[\XXB,\XX](m)\approx m^{1/p}\approx\varphi_{u}[\XXB,\XX](m),\quad m\in\NN.
\]
\end{theorem}

\begin{proof}Set $\XXB=(\xx_{n})_{n=1}^\infty$, $\XXB^*=(\xx_{n}^*)_{n=1}^\infty$, $a=\sup_{n} \Vert \xx_{n}\Vert$ and $b=\sup_{n} \Vert \xx_{n}^*\Vert$.

Let $A\subseteq\NN$ with $|A|= m<\infty$. Use Theorems~\ref{thm:EstimateCC} and ~\ref{thm:CEP} to choose an extension
\[
T_A\colon\ell_{p}\to [\xx_{n} \colon n\in A]
\]
of $S_A[\XXB,\XX]$ with $\Vert T_A\Vert \le C_1(1+\log(m))^{1/p}$ for some constant $C_1$ only depending on $\XXB$. For $n\in A$, let $\yy_{n}^*=\xx_{n}^*\circ T_A$ so that $T_A(f)=\sum_{n\in A} \yy_{n}^*(f) \, \xx_{n}$ for every $f\in\ell_{p}$. We have
\[
\sup_{n\in\NN} \Vert \yy_{n}^*\Vert \le b C_1(1+\log(m))^{1/p}.
\]

By Lemma~\ref{lem:SAlp} the unit vector system in a strongly absolute basis of $\ell_{p}$. Moreover, if $\SA$ denotes its strongly absolute function and
\[
\epsilon_0=\frac{1}{2 a b C_1 (1+\log(m))},
\]
we have
\begin{equation}\label{eq:deltabound}
\delta:=\frac{1}{\SA(\epsilon_0)}\ge \epsilon_0^{p/(1-p)}=\frac{1}{2} (2 a b C_1)^{-p/(1-p)} (1+\log(m))^{-1/(1-p)}.
\end{equation}
Applying Lemma~\ref{lem:WoIdea} to $\SSB=(\xx_{n},\yy_{n}^*)_{n\in A}$ with $C=2$ yields

\begin{align}\label{eq:estimateomega}
m
&\le 2 \sum_{j\in \Omega_\delta(\SSB)} \left| \sum_{i\in A} \yy_{n}^*(\ee_{j}) \, \ee_{j}^*(\xx_{n})\right| \nonumber\\
&\le 2 \sum_{j\in \Omega_\delta(\SSB)} \left\Vert\sum_{i\in A} \yy_{n}^*(\ee_{j}) \,\xx_{n}\right\Vert_{p} \nonumber\\
&=2\sum_{j\in \Omega_\delta(\SSB)} \Vert T_A(\ee_{j})\Vert_{p}\nonumber \\
&\le 2 |\Omega_\delta(\SSB)| \, \Vert T_A\Vert \nonumber \\
&\le 2 b C_1 (1+\log(m))^{1/p} |\Omega_\delta(\SSB)|.
\end{align}

For every $j\in \Omega_\delta(\SSB)$ we have
\begin{equation}\label{eq:estimate33}
\left( \sum_{n\in A} |\ee_{j}^*(\xx_{n})|^2\right)^{1/2} \ge \frac{1}{b} \left( \sum_{n\in A} |\ee_{j}^*(\xx_{n}) \, \xx_{n}^*(\ee_{j}) |^2\right)^{1/2} \ge\frac{\delta}{b}.
\end{equation}
Set $q=(1-p+p^2)/(p^2-p^3)$ and
\[
s_{m}=\frac{m^{1/p} }{(1+\log(m))^{q}}.
\]
Combining Lemma~\ref{lem:LP}, with \eqref{eq:estimateomega}, \eqref{eq:estimate33} and \eqref{eq:deltabound} gives
\begin{align*}
\left\|\sum_{n\in A} \xx_{n}\right\|_{p}
&\ge\frac{1}{C_2}\left( \sum_{j\in \Omega_\delta(\SSB)} \left( \sum_{n\in A} |\ee_{j}^*(\xx_{n})|^2\right)^{p/2}\right)^{1/p}\\
&\ge \frac{\delta}{C_2 b} |\Omega_\delta(\SSB)|^{1/p}\\
& \ge \frac{\delta}{C_2 b} \frac{m^{1/p}}{ (2 b C_1)^{1/p}(1+\log(m))^{1/p^2}}\\
&\ge \left(2 a^{p^2} b C_1^{1-p+p^2} C_2^{p-p^2}\right)^{-1/p(1-p)} s_{m},
\end{align*}
where $C_2$ is a constant depending only on $\XXB$.

Since for $m$ large enough, $\sss=(s_{m})_{m=1}^\infty$ is increasing, there is a weight $\ww$ whose primitive weight is equivalent to $\sss$. By Lemma~\ref{lem:DemocracyEmdedding} \eqref{DemocracyEmdedding:item:1}, the coefficient transform $\Fou$ is a bounded linear operator from $\XX$ into $d_\infty(\ww)$. But $\sum_{n=1}^\infty 1/s_{n}<\infty$, which yields $d_\infty(\ww) \subseteq \ell_1$ and so $\Fou$ is a bounded linear map from $\XX$ into $\ell_1$. With this new piece of information about the coefficient transform we will next be able to use a bootstrap argument for improving the above estimates.

Use again Theorem~\ref{thm:CEP} to determine a constant $C_3$ and linear operators
\[
L_A\colon\ell_{p} \to \ell_1(A), \quad A\subseteq\NN, \, |A|<\infty,
\]
such that $\Vert L_A \Vert \le C_3$ and $L_A(f)=(\xx_{n}^*(f))_{n\in A}$ for every $f\in\XX$. Fix $A\subseteq\NN$ finite and pick $(\zz_{n}^*)_{n\in A}$ in $(\ell_{p})^*$ such that $L_A(f)=(\zz_{n}^*(f))_{n\in A}$ for every $f\in \ell_{p}$. Consider the family $\TTB=(\xx_{n},\zz_{n}^*)_{n\in A}$. Since $\Vert \zz_{n}^* \Vert \le C_3$ for every $n\in A$, Lemmas~\ref{lem:SAlp} and \ref{lem:WoIdea} yield the existence of $\beta>0$ depending only on $p$ and the basis $\XXB$ such that
\[
|A|\le 2 \sum_{j\in \Omega_{\beta}(\TTB)} \left| \sum_{n\in A} \zz_{n}^*(\ee_{j}) \, \ee_{j}^*(\xx_{n})\right|.
\]
With the natural identification of $\ell_\infty(A)$ with the the dual space of $\ell_1(A)$, the dual operator $L_A^*\colon \ell_\infty(A)\to (\ell_{p})^*$ of $L_A$ is given by
\[
L_A^*((a_{n})_{n\in A})(f)=\sum_{n\in A} a_{n} \, \zz_{n}^*(f), \quad a_{n}\in \FF, \, f\in\ell_{p}.
\]
Thus, if we set $g_{j}= (\ee_{j}^*(\xx_{n}))_{n\in A}$ we have
\[
\left|\sum_{n\in A} \zz_{n}^*(\ee_{j}) \, \ee_{j}^*(\xx_{n})\right|
= \left| L_A^*(g_{j}) (\ee_{j})\right|
\le C_3 \Vert g_{j}\Vert_\infty\Vert \ee_{j}\Vert_{p}
\le aC_3.
\]
for all $j\in\NN$. Therefore
\[
|A|\le 2 a C_3 | \Omega_{\beta}(\TTB)|.
\]
Finally, applying again Lemma~\ref{lem:LP} yields
\[
\left\|\sum_{n\in A} \xx_{n}\right\|_{p}\ge \frac{\beta}{C_2 b} |\Omega_\beta(\TTB)|^{1/p}\ge \frac{\beta}{(2a b^p C_2^p C_3)^{1/p} } |A|^{1/p}.
\]
Hence $m^{1/p} \lesssim \varphi_{l}[\XXB,\XX](m)$ for $m\in\NN$. Taking into account inequality~\eqref{eq:ObviouspBanach} the proof is over.
\end{proof}

Let $\LL$ be a quasi-Banach space with a strongly absolute unconditional basis $\BB$. Suppose that $\LL$ has the property that all quasi-greedy bases of $\LL$ are democratic. Then, in particular, $\BB$ is democratic hence greedy. So, in light of Theorem~\ref{MainThm}, in order to enlarge the scant list of quasi-Banach spaces where all quasi-greedy bases are democratic it is natural to look for quasi-Banach spaces with a strongly absolute greedy basis. Among them, Hardy spaces deserve special attention.

\begin{question}Let $0<p<1$ and $d\in\NN$.
Are all quasi-greedy bases in $H_{p}(\Disc^d)$ democratic?
\end{question}

We close this section by applying Theorem~\ref{thm:main} to $\SL_{p}$-spaces for $0<p<1$. Recall that a closed subspace $\XX$ of a quasi-Banach space $\LL$ is said to be \emph{locally complemented} in $\LL$ if there is a constant $C$ such that for every finite-dimensional subspace $\VV$ of $\LL$ and every $\epsilon>0$ there is a linear operator $T\colon \VV\to \XX$ with $\Vert T \Vert \le C$ and $\Vert T|_{\VV\cap \XX} -\Id_{\VV\cap \XX}\Vert \le\epsilon$. Given $0<p\le 1$, following Kalton \cite{Kalton1984} we say that a quasi Banach space is an \emph{$\SL_{p}$-space} if it is isomorphic to a locally complemented subspace of $L_{p}(\mu)$ for some measure $\mu$.

\begin{corollary}\label{cor:RTOPSLp}Let $0<p<1$. Suppose $\XXB$ is a quasi-greedy basis of a $\SL_{p}$-space $\XX$ with the bounded approximation property. Then:
\begin{enumerate}[(i)]

\item\label{cor:item1} $\varphi_{l}[\XXB,\XX](m)\approx\varphi_{u}[\XXB,\XX](m)\approx m^{1/p}$ for $m\in\NN$. In particular, $\XXB$ is democratic.

\item\label{cor:item2} $\ell_{p} \stackrel{\XXB}\hookrightarrow\XX \stackrel{\XXB}\hookrightarrow \ell_{p,\infty} $.

\item\label{cor:item3} For $p<q\le 1$, the $q$-Banach envelope of $\XXB$ is equivalent to the unit vector system of $\ell_{q}$.

\item\label{cor:item5} For $p<q\le 1$, the $q$-Banach envelope of $\XX$ is isomorphic to $\ell_{q}$.

\item\label{cor:item6} The dual space $\XX^*$ is isomorphic to $\ell_\infty$.

\item\label{cor:item4} The dual basis of $\XXB$ is equivalent to the unit vector system of $c_0$.

\end{enumerate}

\end{corollary}

\begin{proof}By \cite{Kalton1984}*{Theorem 6.4} we can suppose without loss of generality that $\XX$ is a locally complemented subspace of $\ell_{p}$. By \cite{Kalton1984}*{Theorem 3.4}, $\XX$ has the compact extension property in $\ell_{p}$. Then, \eqref{cor:item1} follows from \cite{AABW2019}*{Theorem 3.13} and Theorem~\ref{thm:main}. Once we have proved \eqref{cor:item1}, \eqref{cor:item2} is a consequence of \cite{AABW2019}*{Theorem 8.12}, and \eqref{cor:item3} and \eqref{cor:item5} follow from \cite{AABW2019}*{Proposition 9.12}. In turn, \eqref{cor:item6} is a consequence of \eqref{cor:item5}.
Finally, \eqref{cor:item4} follows from combining \eqref{cor:item3} with \cite{AABW2019}*{Corollary 9.10}.
\end{proof}

\begin{theorem}Let $0<p<1$. If $\XXB$ is a quasi-greedy basis of an $\SL_{p}$-space then $\XXB$ is almost greedy.
\end{theorem}

\begin{proof}Just combine \cite{AABW2019}*{Theorem 3.13}, Corollary~\ref{cor:RTOPSLp}~\eqref{cor:item1} and \cite{AABW2019}*{Theorem 5.3}.
\end{proof}

\section{Appendix. Democracy of quasi-greedy bases of $\ell_1$}\label{Sect:DemL1}
\noindent
A Banach space $\XX$ is called a \emph{GT space} \cite{Pisier1986} if every bounded linear operator $T\colon \XX\to\ell_2$ is absolutely summing, i.e., there is a constant $C$ such that for all finite collections $(f_{k})_{k\in B}$ in $\XX$,
\begin{equation}\label{ASE}
\sum_{k\in B} \Vert T(f_{k}) \Vert_2
\le C \sup_{|\varepsilon_{k}|=1 } \left\Vert \sum_{k\in B} \varepsilon_{k} f_{k}\right\Vert.
\end{equation}
The smallest constant $C$ such that \eqref{ASE} holds is the absolutely summing norm of $T$ and is denoted by $\pi_1(T)$. Of course, if $\XX$ is a GT space then there is a constant $D$ such that $\pi_1(T)\le D \Vert T \Vert$ for all $T\in\LO(X,\ell_2)$.

Given $1\le p\le \infty$, following \cite{LinPel1968} we say that an infinite-dimensional Banach space $\XX$ is a \emph{$\SL_{p}$-space} if there is $\lambda\ge 1$ such that for every finite dimensional subspace $E\subseteq \XX$ there is $d\in\NN$ and an $d$-dimensional subspace $E\subseteq F\subseteq \XX$ such that $d(F,\ell_{p}^d)\le \lambda$. We could extend this definition to non-locally convex spaces, but at the end we shall conclude that if a quasi-Banach is a $\SL_{p}$-space for some $p\ge 1$ then it is (isomorphic to) a Banach space. It is known \cite{Kalton1984} that, for $p=1$, this definition coincides with the one we gave in Section~\ref{sect:main}.

Lindenstrauss and Pe{\l}czy{\'n}ski \cite{LinPel1968} reinterpreted Grothendieck's inequality \cite{GT1953} in the following fashion.
\begin{theorem}[\cite{LinPel1968}*{Theorem 4.1}] Every $\SL_1$-space is a GT space.
\end{theorem}
Other examples of GT spaces besides $\SL_1$-spaces are the spaces $L_1/H$, where $H$ is a subspace of $L_1$ isomorphic to a Hilbert space (see \cite{Pisier1986}*{Corollary 6.11}), the dual of the disc algebra, and $L_1/H_1$ (see \cite{Bourgain1984}).

Our approach to the proof of Theorem~\ref{thm:GTDemocratic} relies on the following lemma and a bootstrap argument.

\begin{lemma} \label{FeedBack} Let $\XXB$ be a basis of a GT space $\XX$. Assume that $\XXB$ has the bounded restricted truncation operator property. Let $\sss=(s_{m})_{m=1}^\infty$ be a sequence of positive numbers such that $s_{m} \lesssim \varphi_{l}[\XXB,\XX](m)$ for $m\in\NN$. Then
\[
t_{m}:={m}{ \left(\sum_{n=1}^m \frac{1}{s_{n}^2}\right)^{-1/2}}
\lesssim \varphi_{l}[\XXB,\XX](m),\quad m\in \NN.
\]
\end{lemma}

\begin{proof}The basis $\XXB$ is in particular SUCC. Let $C$ be as in \eqref{eq:succ}.
By Lemma~\ref{lem:DemocracyEmdedding} \eqref{DemocracyEmdedding:item:1}, the coefficient transform $\Fou$ with respect to $\XXB=(\xx_{j})_{j=1}^\infty$ is bounded from $\XX$ into the weak Lorentz space $d_\infty(\Delta\sss)$. Let us denote by $C_1$ its norm. Let $A\subseteq\NN$ with $|A|=m$. The coordinate projection from $\FF^\NN$ onto $\FF^A$ is, when regarded as an operator from $d_\infty(\Delta\sss)$ onto $\ell_2(A)$, bounded by $m/t_{m}$. Indeed, given $f=(a_{n})_{n=1}^\infty\in d_\infty(\Delta\sss)$,
\[
\sum_{j\in A} |a_{j}|^2 \le \sum_{n=1}^m (a_{n}^*)^2\le \Vert f\Vert_{\infty,\Delta\sss}^2 \sum_{n=1}^m \frac{1}{s_{n}^2}
= \Vert f\Vert_{\infty,\Delta\sss}^2 \frac{m^2}{t_{m}^2}.
\]
Consequently, the operator
\[
T_{A}\colon \XX \to \ell_2(A), \quad f\mapsto (\xx_{j}^*(f))_{j\in A}
\]
satisfies $\Vert T_A\Vert \le C_1 m/t_{m}$. Since $\XX$ is a GT space, for every finite family $(f_{k})_{k\in B}$ in $\XX$ we have
\begin{equation}\label{improved}
\sum_{k\in B} \Vert T_{A}(f_{k}) \Vert_2 \le C_1D \frac{m}{t_{m}}
\sup_{|\varepsilon_{k}|=1} \left\Vert \sum_{k\in B} \varepsilon_{k} f_{k}\right\Vert,
\end{equation}
for some constant $D$ depending only of $\XX$.
Applying \eqref{improved} to $(\xx_{k})_{k\in A}$ we obtain
\[
m \le C_1 D\frac{m}{t_{m}} \sup_{|\varepsilon_{k}|=1} \left\Vert \sum_{k\in A} \varepsilon_{k} \,\xx_{k}\right\Vert
\le C C_1 D\frac{m}{t_{m}} \left\Vert \sum_{k\in A} \xx_{k}\right\Vert,
\]
and so inequality~\eqref{eq:LDA} completes the proof.
\end{proof}

Next, we state and prove a result slightly more general than Theorem~\ref{thm:GTDemocratic}.
\begin{theorem}\label{QGtoAG} Let $\XXB$ basis of a GT space $\XX$. Suppose that $\XXB$ has the bounded restricted truncation operator property. Then $\XXB$ is democratic. Moreover,
\[
\varphi_{l}[\XXB,\XX](m) \approx m \approx \varphi_{u}[\XXB,\XX](m),\quad m\in\NN.
\]
\end{theorem}

\begin{proof}
Since $\XXB$ is a SUCC basis, inequality~\eqref{eq:LDA} yields
$1\lesssim \varphi_{l}[\XXB,\XX](m)$ for $m\in\NN$. We feed Lemma~\ref{FeedBack} with $s_{m}=1$ for all $m\in\NN$ and obtain
\[
m^{1/2}\lesssim \varphi_{l}[\XXB,\XX](m), \quad m\in\NN.
\]
Let $H_{m}=\sum_{n=1}^m 1/n$ for $m\in\NN$. Applying again Lemma~\ref{FeedBack}, now with $s_{m}=m^{1/2}$ for all $m\in\NN$, gives
\[
H_{m}^{-1/2} m \lesssim \varphi_{l}[\XXB,\XX](m), \quad m\in\NN.
\]
Since $\sum_{n=1}^{\infty} H_{n} n^{-2}<\infty$, using once more Lemma~\ref{FeedBack} gives
\[
m \lesssim \varphi_{l}[\XXB,\XX](m), \quad m\in\NN.
\]
Appealing to inequality~\eqref{eq:ObviouspBanach} the proof is over.
\end{proof}

We close with some applications. For the reader's convenience we will state our results for quasi-greedy bases, but they also hold for bases with the bounded restricted truncation operator property.

\begin{corollary} \label{embeddingGT} Let $\XXB$ be quasi-greedy basis in a GT space $\XX$. Then
$\ell_1 \stackrel{\XXB}\hookrightarrow\XX \stackrel{\XXB}\hookrightarrow \ell_{1,\infty} $.
\end{corollary}

\begin{proof} Just combine Lemma~\ref{QGtoAG} with Lemma~\ref{lem:DemocracyEmdedding}.
\end{proof}

\begin{corollary}\label{Equivalencelp} Suppose that a sequence $\XXB$ in $\ell_1$ is a quasi-greedy basis of both spaces $\ell_1$ and $\ell_2$. Then $\XXB$ is equivalent to the canonical basis of $\ell_{p}$ for all $1<p<2$.\end{corollary}
\noindent

\begin{proof} Combining Corollary~\ref{embeddingGT} with the embeddings in \eqref{eq:HilbertEmbeddings} and Marcinkiewicz's interpolation theorem (see \cite{BS1988}) we obtain
\[
\ell_{p}\stackrel{\XXB}\hookrightarrow \ell_{p} \stackrel{\XXB}\hookrightarrow \ell_{p},
\]
which yields the desired conclusion.
\end{proof}

Our last  results show that the reason for leaving out $L_1$ in Theorem~\ref{thm:UBQGOrtho} is not due to a limitation of our methods but to the geometric structure of the space.

\begin{corollary}Let $\mu$ be a non-purely atomic measure. There is no family of functions that is simultaneously a quasi-greedy basis in both spaces $L_1(\mu)$ and $L_2(\mu)$.\end{corollary}
\begin{proof} Suppose that $\XXB$ is a quasi-greedy basis in both $L_1(\mu)$ and $L_2(\mu)$. In particular, $L_1(\mu)$ and $L_2(\mu)$ are separable Banach spaces. Let $1<p<2$. Combining Corollary~\ref{embeddingGT} and the embedding \eqref{eq:HilbertEmbeddings} with Marcinkiewicz's interpolation theorem we obtain
\[
\ell_{p}\stackrel{\XXB}\hookrightarrow L_{p}(\mu) \stackrel{\XXB}\hookrightarrow \ell_{p}.
\]
Therefore $L_{p}(\mu)\simeq \ell_{p}$, an absurdity because $L_{p}(\mu)\simeq L_{p}([0,1])\not\simeq\ell_{p}$ (see \cite{JL2001}).
\end{proof}

\begin{corollary}Let $\mu$ be finite measure and $\Psi=(\psib_{n})_{n=1}^\infty$ be a quasi-greedy basis of $L_1(\mu)$. Then, for  $1<q\le\infty$,
\[
\lim\limits_{n\to\infty} \Vert \bm{\psib}_{n}\Vert_{q}=\infty.
\]

\end{corollary}
\begin{proof}Suppose by contradiction that $\liminf_{n} \Vert \psib_{n}\Vert_{q}<\infty$ for some $q>1$. Then there is a subbasis  $\Psi_0=(\psib_{n_{k}})_{k=1}^\infty$ of $\Psi$ with $\sup_{k} \Vert\psib_{n_{k}}\Vert_{q}<\infty$. Since $\Psi_0$ is a quasi-greedy basic sequence, combining Theorem~\ref{QGtoAG}  with \cite{AACV2019}*{Lemma 2.3} yields
\[
m\lesssim \varphi_{l}[\Psi,L_1(\mu)](m)
\le\varphi_{u}[\Psi_0,L_1(\mu)](m)
\lesssim m^{1/q}, \quad m\in\NN,
\]
so that $\sup_{m} m^{1-1/q}<\infty$, an absurdity.
\end{proof}


\end{document}